\documentclass[reqno,12pt,a4paper]{amsart}

\usepackage{mathrsfs}
\usepackage{amssymb}
\usepackage{amsfonts}
\usepackage{latexsym}
\usepackage{amsthm}
\usepackage{graphicx}
\usepackage{xcolor}

\theoremstyle{plain}
\newtheorem{theorem}{Theorem}[section]

\newtheorem{proposition}[theorem]{Proposition}

\theoremstyle{remark}
\newtheorem{remark}[theorem]{\bf Remark}
\newtheorem{example}[theorem]{\bf Example}

\newcommand{\lb}{\label}
\newcommand{\er}{\eqref}

\newcommand{\qqq}{\qquad}
\newcommand{\qq}{\quad}
\newcommand{\no}{\noindent}

\newcommand{\s}{\sigma}
\newcommand{\n}{\nu}
\newcommand{\vk}{\varkappa}
\newcommand{\ve}{\varepsilon}
\newcommand{\vp}{\varphi}
\newcommand{\g}{\gamma}

\newcommand{\D}{\Delta}

\newcommand{\cA}{{\mathcal A}}

\newcommand{\cC}{{\mathcal C}}
\newcommand{\cE}{{\mathcal E}}
\newcommand{\cG}{{\mathcal G}}

\newcommand{\cN}{{\mathcal N}}

\newcommand{\cS}{{\mathcal S}}

\newcommand{\cV}{{\mathcal V}}

\newcommand{\C}{{\mathbb C}}
\newcommand{\N}{{\mathbb N}}
\newcommand{\R}{{\mathbb R}}
\newcommand{\T}{{\mathbb T}}
\newcommand{\Z}{{\mathbb Z}}

\newcommand{\be}{{\textbf e}}
\newcommand{\bc}{{\textbf c}}

\newcommand{\gp}{\mathfrak{p}}

\newcommand{\gS}{\mathfrak{S}}

\newcommand{\mn}{\mathrm n}

\newcommand{\bu}{\bullet}
\newcommand{\ts}{\times}
\newcommand{\sm}{\setminus}

\newcommand{\ol}{\overline}

\newcommand{\lan}{\langle}
\newcommand{\ran}{\rangle}

\newcommand{\iy}{\infty}
\newcommand{\pa}{\partial}

\def\diag{\mathop{\mathrm{diag}}\nolimits}
\def\Im{\mathop{\mathrm{Im}}\nolimits}

\renewcommand{\[}{\begin{equation}}
\renewcommand{\]}{\end{equation}}

\def\l{\lambda}

\def\t{\tau}
\def\o{\omega}
\def\ss{\subset}

\begin{document}
\makeatletter
\@namedef{subjclassname@2020}{\textup{2020} Mathematics Subject Classification}
\makeatother

\title[On the measure of spectra for discrete Schr\"odinger operators]{On the measure of spectra for discrete Schr\"odinger operators on periodic graphs}

\date{\today}
\author[Natalia Saburova]{Natalia Saburova}
\address{Northern (Arctic) Federal University, Severnaya Dvina Emb. 17, Arkhangelsk, 163002, Russia, \ n.saburova@gmail.com, \ n.saburova@narfu.ru}


\begin{abstract}
We consider discrete Schr\"odinger operators $H_{\mu Q}=\D+\mu Q$ with real periodic potentials $Q$ on periodic graphs, where $\D$ is the adjacency operator and $\mu\in\R$ is a coupling constant. The spectra of the operators consist of a finite number of closed intervals (bands). In the large coupling regime, we obtain an asymptotic upper bound for the measure of the spectrum of $H_{\mu Q}$ which depends essentially on a "degeneracy degree" of the potential $Q$. This result extends the result of Y. Last obtained for the one-dimensional lattice $\Z$ to the case of general periodic graphs. It also may serve as a certain quantitative complement to the recent criterion of J.~Fillman for the measure of the spectrum of $H_{\mu Q}$ to go to zero as $\mu\to\iy$.
\end{abstract}

\maketitle

\section{Introduction}
\setcounter{equation}{0}

We consider discrete Schr\"odinger operators $H_Q=\D+Q$ with real periodic potentials $Q$ on $\Z^d$-periodic graphs (for simplicity) embedded into $\R^d$, where $\D$ is the adjacency operator. In solid state physics, the  spectra of these operators determine main physical properties of materials. The spectra $\s(H_Q)$ of the Schr\"odinger operators $H_Q$ consist of a finite number of closed intervals (bands) which may be separated by gaps (forbidden zones). The spectral bands correspond to the set of wave energies at which the waves can propagate through the medium. The problem of estimating the measure $|\s(H_Q)|$ of the spectrum of the Schr\"odinger operators $H_Q$ on periodic graphs was studied in \cite{KS14} -- \cite{KS22}. It is a rather difficult problem to obtain estimates which involve not only geometric parameters of the periodic graph but also the periodic potential itself, and there are only some (quite rude but, as we will see, important) lower bounds depending on the potential in \cite{KS22}. Unfortunately, the upper bounds presented in \cite{KS22} do not depend on the potential at all, i.e., they do note give any information about the influence of the potential on the measure of the spectrum.

This paper was motivated by the following recent result (for detailed definitions see Section~\ref{Sec2}):

\begin{theorem}\lb{FT11}
\cite[Theorem 1.1]{F26} Assume $Q:\cV\to\R$ is a periodic potential on a $\Z^d$-periodic graph $\cG=(\cV,\cE)$. Then
$$
\lim_{\mu\to\iy}\big|\s(H_{\mu Q})\big|=0
$$
if and only if $\cG$ does not contain a path with infinitely many distinct vertices such that the potential is constant along this path.
\end{theorem}

The latter condition in Theorem \ref{FT11} is equivalent to the absence of a homotopically nontrivial cycle in the \emph{quotient graph} $\cG_*=\cG/\Z^d$ (which is a graph on the torus $\R^d/\Z^d$) with the same potential values at its vertices. The author wrote \cite[p.2]{F26} that "adapting the techniques of \cite{AFFL25} one can show that for \emph{non-degenerate} $Q$ (i.e., for $Q$ injective on $\cG_*$), the decay in Theorem~\ref{FT11} happens at least at the rate $\mu^{1-\g}$, where $\g$ is the shortest \emph{length} (the number of edges) of a homotopically nontrivial cycle in $\cG_*$".

In this paper, in the large coupling regime, we obtain an asymptotic upper bound for the measure of the spectrum of $H_{\mu Q}$ when the potential $Q$ has a \emph{partial degeneracy} along cycles of the quotient graph $\cG_*$. We note that for the Schr\"odinger operator $H_{\mu Q}$ on the one-dimensional lattice $\Z$, such decay bound was obtained by Y.~Last:
\begin{theorem}
\cite[Theorem 3]{L92} Let $Q:\Z\to\R$ be a periodic potential, and let $\mu$ be a positive coupling constant, then for the potential $\mu Q$, in the limit $\mu\to\iy$:
\[\lb{Laas}
\big|\s(H_{\mu Q})\big|=O(\mu^{1-\g}),
\]
where $\g$ is defined by:
\[\lb{Lada}
\begin{aligned}
&\g=\min_{a\in Q(\Z)}d(a),\\
&d(a)=\max\{j-k\mid j>k, \; Q(j)=Q(k)=a, \; Q(l)\neq a,\; \forall\,j>l>k\},
\end{aligned}
\]
i.e., $\g$ is the minimum of the distances $d(a)$, where $d(a)$ is the maximal distance between two potential points with the same value $a$.
\end{theorem}

Let us formulate our main result. We consider the Schr\"odinger operator $H_{\mu Q}$ with a real $\Z^d$-periodic potential $Q$ on a $\Z^d$-periodic graph $\cG\ss\R^d$ with the quotient graph $\cG_*=(\cV_*,\cE_*)=\cG/\Z^d\ss\R^d/\Z^d$. Let $a\in Q(\cV_*)$, i.e., $a=Q(v)$ for some $v\in\cV_*$, and let
\[\lb{cVa}
\cV_*^a=\{v\in\cV_* : Q(v)=a\}\subseteq\cV_*.
\]
Denote by $d(u,\cV_*^a)$ the distance from a vertex $u\in\cV_*$ to the vertex set $\cV_*^a$:
\[\lb{dvse}
d(u,\cV_*^a)=\min\limits_{v\in \cV_*^a}d(u,v),
\]
where $d(u,v)$ is the \emph{distance} between vertices $u$ and $v$, i.e., the minimal \emph{length} (the number of edges) of a path from $u$ to $v$, in the quotient graph $\cG_*$. If $u\in\cV_*^a$, then $d(u,\cV_*^a)=0$. For each edge $\be$ of $\cG_*$, we define the weight $\o_a(\be)$ by
\[\lb{wae}
\o_a(\be)=d(u,\cV_*^a)+d(v,\cV_*^a)\geq0,\qqq \be=\{u,v\}\in\cE_*.
\]
Note that $\big|d(u,\cV_*^a)-d(v,\cV_*^a)\big|\in\{0,1\}$ for each $\be=\{u,v\}\in\cE_*$.

\begin{theorem}\lb{MTh}
Let $\mu$ be a positive coupling constant. Then, in the limit $\mu\to\iy$, the measure of the spectrum of the Schr\"odinger operator $H_{\mu Q}$ on a periodic graph $\cG$ satisfies
\[\lb{mare}
\big|\s(H_{\mu Q})\big|=O(\mu^{-\o}),
\]
where
\[\lb{ooa}
\o=\min\limits_{a\in Q(\cV_*)}\o(a),\qqq \o(a)=\min\limits_{\bc\in\cC^+}\max\limits_{\be\in\bc}\o_a(\be).
\]
Here $\cC^+$ is the set of all homotopically nontrivial cycles in the quotient graph $\cG_*=(\cV_*,\cE_*)$, the maximum is taken over all edges $\be$ of the cycle $\bc$, and $\o_a(\be)$ is the weight of the edge $\be\in\cE_*$ defined by \er{wae}.
\end{theorem}

\begin{remark}\lb{reMT}
\emph{i}) Recall that the quotient graph $\cG_*$ is a graph on the $d$-dimensional torus $\R^d/\Z^d$. A \emph{homotopically nontrivial} cycle in $\cG_*$ is a cycle on the torus that cannot be continuously deformed to a single point (for detailed definitions see Section \ref{ssci}). Note that the set of all (simple) cycles in the \emph{finite} quotient graph $\cG_*$ is finite, and $\cC^+\neq\varnothing$ (there are at least $d$ homotopically nontrivial cycles in $\cG_*$; they provide the connectivity of the $\Z^d$-periodic graph~$\cG$, see Fig.\ref{FBri}).

\emph{ii}) For each $a\in Q(\cV_*)$, there are $\#\cV_*^a$ spectral bands $\s_{\mu a}$ of $H_{\mu Q}$ "near" $\mu a$, where $\cV_*^a$ is defined by \er{cVa}, and $\#$ denotes the number of elements in a set. The proof of Theorem \ref{MTh} gives that the width of each of those bands
\[\lb{Aswb}
|\s_{\mu a}|=O\big(\mu^{-\o(a)}\big),
\]
where $\o(a)$ is given in \er{ooa}.

\emph{iii}) For the lattice $\Z$ the quotient graph $\cG_*=(\cV_*,\cE_*)$ is a homotopically nontrivial cycle with $\n$ vertices, where $\n$ is the period of the potential $Q$. Then, for each $a\in Q(\cV_*)$, we have
$$
\begin{aligned}
&\o(a)=\max\limits_{\be\in\cE_*}\o_a(\be)=
\max\limits_{\{u,v\}\in\cE_*}(d(u,\cV_*^a)+d(v,\cV_*^a))=d(a)-1,\\
& \o=\min\limits_{a\in Q(\cV_*)}\o(a)=\min\limits_{a\in Q(\cV_*)}d(a)-1=\g-1,
\end{aligned}
$$
where $d(a)$ and $\g$ are given in \er{Lada}. Thus, the asymptotics \er{mare} for the one-dimensional lattice $\Z$ coincides with Last's result \er{Laas} and extends this result to the case of general periodic graphs.

\emph{iv}) If there is a cycle $\bc\in\cC^+$ with a constant potential $a$, then the weight $\o_a(\be)=0$ for each edge $\be$ of this cycle $\bc$. This yields that $\o=\o(a)=0$ and $\big|\s(H_{\mu Q})\big|=O(1)$.

If there are no cycles in $\cC^+$ with a constant potential, then for any $\bc\in\cC^+$ and any $a\in Q(\cV_*)$, there exists an edge $\be\in\bc$ such that $\o_a(\be)\geq1$. Therefore, $\o\geq1$ and $\big|\s(H_{\mu Q})\big|\to 0$ as $\mu\to\iy$. In particular, if the potential $Q$ is non-degenerate, then $\o=\g-1$, where $\g$ is the shortest length of a cycle in $\cC^+$.

Thus, the asymptotics \er{mare} agrees with Theorem \ref{FT11} \cite{F26} and can be considered as a certain quantitative complement to that result.

\emph{v}) The asymptotic upper bound \er{mare} means that $\big|\s(H_{\mu Q})\big|$ decays at the rate $\mu^{-\o}$ or faster. Using the same techniques as in \cite{KS22}, we can obtain the following lower bound for the total bandwidth $\gS(H_{\mu Q})$ of $H_{\mu Q}$ (the sum of widths of all its bands, which may be greater than the measure of $\s(H_{\mu Q})$, since the bands may overlap):
\[\lb{tobw}
\gS(H_{\mu Q})\geq\frac{\cN_\g^+\,\mu^{1-\g}}{\g\big(q+\frac{\vk}\mu\big)^{\g-1}}\,,\qqq \textrm{where}\qqq q=\max_{v\in\cV_*}Q(v)-\min_{v\in\cV_*}Q(v),
\]
$\vk$ is the maximum vertex degree, $\g$ is the shortest length of a cycle $\bc\in\cC^+$, and $\cN_\g^+$ is the number of all cycles of length $\g$ from $\cC^+$, see \cite[Corollary 4]{KS22}. Thus, at least for non-degenerate potentials, the asymptotic upper bound \er{mare} is sharp. For degenerate potentials the lower estimate \er{tobw} guaranties that the rate of decay for $\big|\s(H_{\mu Q})\big|$ is not faster than $\mu^{1-\g}$. It is an open problem to obtain a sharp asymptotic bound for degenerate potentials (or to show that the bound \er{mare} is sharp in any case).

\emph{vi}) The proof of Theorem \ref{FT11} \cite{F26} is based on the interpretation of the first order terms in the perturbation series for the Floquet eigenvalues of $H_{\mu Q}=\D+\mu Q$ as Floquet eigenvalues of $\D$ on some subgraph of the original periodic graph $\cG$. The proof of Theorem \ref{MTh} does not use this idea and is close to the proof from \cite{L92} (adapted to the case of general periodic graphs). Our proof consists of the following steps:

$\bu$ we estimate the width of each spectral band in terms of the Floquet eigenvectors of $H_{\mu Q}$ (Proposition \ref{Pubb});

$\bu$ in the large coupling regime, we construct the perturbation series for these Floquet eigenvectors of $H_{\mu Q}$, using spectral projections;

$\bu$ finally, we choose appropriate embeddings of the periodic graph into $\R^d$ (Proposition \ref{Ppge}) which allow us to construct the asymptotic upper bound \er{mare}.
\end{remark}

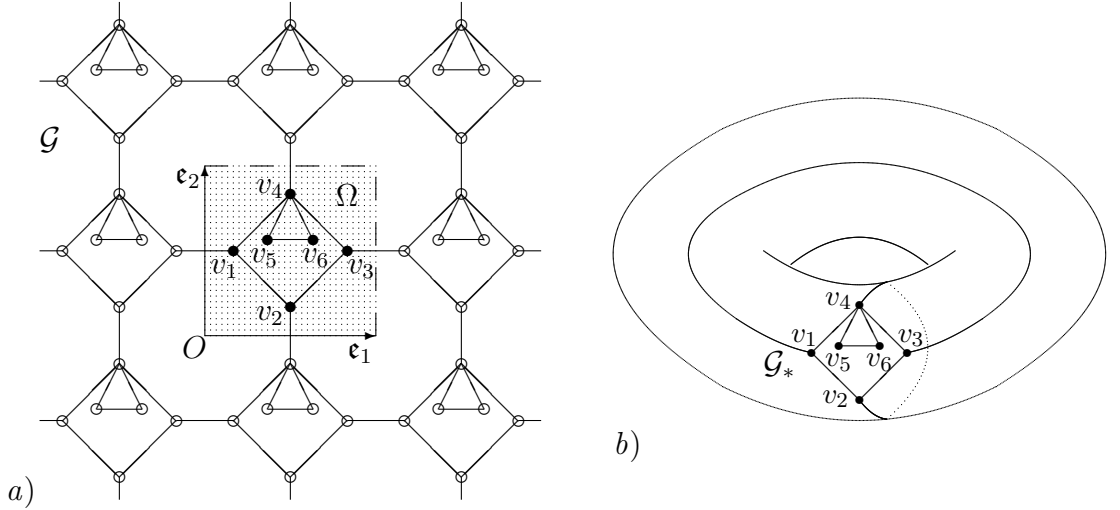
\begin{figure}[t]
\centering
\unitlength 1.5mm 
\linethickness{0.4pt}

\begin{picture}(59,50)(0,0)
\put(5,8.0){\emph{a})}
\put(8,39.0){$\cG$}
\put(22.5,22.5){\vector(1,0){15.00}}
\put(22.5,22.5){\vector(0,1){15.00}}
\multiput(22.5,37.5)(4,0){4}{\line(1,0){2}}
\multiput(37.5,22.5)(0,4){4}{\line(0,1){2}}

\bezier{30}(22.5,22.5)(22.5,30)(22.5,37.5)
\bezier{30}(23.0,22.5)(23.0,30)(23.0,37.5)
\bezier{30}(23.5,22.5)(23.5,30)(23.5,37.5)
\bezier{30}(24.0,22.5)(24.0,30)(24.0,37.5)
\bezier{30}(24.5,22.5)(24.5,30)(24.5,37.5)
\bezier{30}(25.0,22.5)(25.0,30)(25.0,37.5)
\bezier{30}(25.5,22.5)(25.5,30)(25.5,37.5)
\bezier{30}(26.0,22.5)(26.0,30)(26.0,37.5)
\bezier{30}(26.5,22.5)(26.5,30)(26.5,37.5)
\bezier{30}(27.0,22.5)(27.0,30)(27.0,37.5)
\bezier{30}(27.5,22.5)(27.5,30)(27.5,37.5)
\bezier{30}(28.0,22.5)(28.0,30)(28.0,37.5)
\bezier{30}(28.5,22.5)(28.5,30)(28.5,37.5)
\bezier{30}(29.0,22.5)(29.0,30)(29.0,37.5)
\bezier{30}(29.5,22.5)(29.5,30)(29.5,37.5)
\bezier{30}(30.0,22.5)(30.0,30)(30.0,37.5)
\bezier{30}(30.5,22.5)(30.5,30)(30.5,37.5)
\bezier{30}(31.0,22.5)(31.0,30)(31.0,37.5)
\bezier{30}(31.5,22.5)(31.5,30)(31.5,37.5)
\bezier{30}(32.0,22.5)(32.0,30)(32.0,37.5)
\bezier{30}(32.5,22.5)(32.5,30)(32.5,37.5)
\bezier{30}(33.0,22.5)(33.0,30)(33.0,37.5)
\bezier{30}(33.5,22.5)(33.5,30)(33.5,37.5)
\bezier{30}(34.0,22.5)(34.0,30)(34.0,37.5)
\bezier{30}(34.5,22.5)(34.5,30)(34.5,37.5)
\bezier{30}(35.0,22.5)(35.0,30)(35.0,37.5)
\bezier{30}(35.5,22.5)(35.5,30)(35.5,37.5)
\bezier{30}(36.0,22.5)(36.0,30)(36.0,37.5)
\bezier{30}(36.5,22.5)(36.5,30)(36.5,37.5)
\bezier{30}(37.0,22.5)(37.0,30)(37.0,37.5)
\bezier{30}(37.5,22.5)(37.5,30)(37.5,37.5)
\put(35,20.5){$\mathfrak{e}_1$}
\put(19.9,36){$\mathfrak{e}_2$}
\put(20.5,20,5){$O$}
\put(34,34){$\Omega$}
\put(10,15){\line(1,-1){5.00}}
\put(10,15){\line(1,1){5.00}}
\put(20,15){\line(-1,-1){5.00}}
\put(20,15){\line(-1,1){5.00}}
\put(20,15){\line(1,0){5.00}}
\put(15,20){\line(0,1){5.00}}
\put(10,15){\line(-1,0){2.00}}
\put(10,30){\line(-1,0){2.00}}
\put(10,45){\line(-1,0){2.00}}
\put(50,15){\line(1,0){2.00}}
\put(50,30){\line(1,0){2.00}}
\put(50,45){\line(1,0){2.00}}
\put(15,10){\line(0,-1){2.00}}
\put(30,10){\line(0,-1){2.00}}
\put(45,10){\line(0,-1){2.00}}
\put(15,50){\line(0,1){2.00}}
\put(30,50){\line(0,1){2.00}}
\put(45,50){\line(0,1){2.00}}

\put(10,15){\circle{1.0}}
\put(15,10){\circle{1.0}}
\put(20,15){\circle{1.0}}
\put(15,20){\circle{1.0}}
\put(17,16){\circle{1}}
\put(13,16){\circle{1}}
\put(15,20){\line(-1,-2){2.00}}
\put(15,20){\line(1,-2){2.00}}
\put(13,16){\line(1,0){4.00}}
\put(25,15){\line(1,-1){5.00}}
\put(25,15){\line(1,1){5.00}}
\put(35,15){\line(-1,-1){5.00}}
\put(35,15){\line(-1,1){5.00}}
\put(35,15){\line(1,0){5.00}}
\put(30,20){\line(0,1){5.00}}
\put(25,15){\circle{1}}
\put(30,10){\circle{1}}
\put(35,15){\circle{1}}
\put(30,20){\circle{1}}
\put(32,16){\circle{1}}
\put(28,16){\circle{1}}
\put(30,20){\line(-1,-2){2.00}}
\put(30,20){\line(1,-2){2.00}}
\put(28,16){\line(1,0){4.00}}
\put(40,15){\line(1,-1){5.00}}
\put(40,15){\line(1,1){5.00}}
\put(50,15){\line(-1,-1){5.00}}
\put(50,15){\line(-1,1){5.00}}
\put(45,20){\line(0,1){5.00}}
\put(40,15){\circle{1}}
\put(45,10){\circle{1}}
\put(50,15){\circle{1}}
\put(45,20){\circle{1}}
\put(47,16){\circle{1}}
\put(43,16){\circle{1}}
\put(45,20){\line(-1,-2){2.00}}
\put(45,20){\line(1,-2){2.00}}
\put(43,16){\line(1,0){4.00}}

\put(10,30){\line(1,-1){5.00}}
\put(10,30){\line(1,1){5.00}}
\put(20,30){\line(-1,-1){5.00}}
\put(20,30){\line(-1,1){5.00}}
\put(20,30){\line(1,0){5.00}}
\put(15,35){\line(0,1){5.00}}
\put(10,30){\circle{1}}
\put(15,25){\circle{1}}
\put(20,30){\circle{1}}
\put(15,35){\circle{1}}
\put(17,31){\circle{1}}
\put(13,31){\circle{1}}
\put(15,35){\line(-1,-2){2.00}}
\put(15,35){\line(1,-2){2.00}}
\put(13,31){\line(1,0){4.00}}
\put(25,30){\line(1,-1){5.00}}
\put(25,30){\line(1,1){5.00}}
\put(35,30){\line(-1,-1){5.00}}
\put(35,30){\line(-1,1){5.00}}
\put(35,30){\line(1,0){5.00}}
\put(30,35){\line(0,1){5.00}}
\put(23,28){$v_1$}
\put(35,28){$v_3$}
\put(27,23.9){$v_2$}
\put(27,35.1){$v_4$}
\put(26.5,29){$v_5$}
\put(31,29){$v_6$}
\put(32,31){\circle*{1}}
\put(28,31){\circle*{1}}
\put(30,35){\line(-1,-2){2.00}}
\put(30,35){\line(1,-2){2.00}}
\put(28,31){\line(1,0){4.00}}
\put(25,30){\circle*{1}}
\put(30,25){\circle*{1}}
\put(35,30){\circle*{1}}
\put(30,35){\circle*{1}}
\put(40,30){\line(1,-1){5.00}}
\put(40,30){\line(1,1){5.00}}
\put(50,30){\line(-1,-1){5.00}}
\put(50,30){\line(-1,1){5.00}}
\put(45,35){\line(0,1){5.00}}
\put(40,30){\circle{1}}
\put(45,25){\circle{1}}
\put(50,30){\circle{1}}
\put(45,35){\circle{1}}
\put(47,31){\circle{1}}
\put(43,31){\circle{1}}
\put(45,35){\line(-1,-2){2.00}}
\put(45,35){\line(1,-2){2.00}}
\put(43,31){\line(1,0){4.00}}
\put(10,45){\line(1,-1){5.00}}
\put(10,45){\line(1,1){5.00}}
\put(20,45){\line(-1,-1){5.00}}
\put(20,45){\line(-1,1){5.00}}
\put(20,45){\line(1,0){5.00}}
\put(10,45){\circle{1}}
\put(15,40){\circle{1}}
\put(20,45){\circle{1}}
\put(15,50){\circle{1}}
\put(17,46){\circle{1}}
\put(13,46){\circle{1}}
\put(15,50){\line(-1,-2){2.00}}
\put(15,50){\line(1,-2){2.00}}
\put(13,46){\line(1,0){4.00}}
\put(25,45){\line(1,-1){5.00}}
\put(25,45){\line(1,1){5.00}}
\put(35,45){\line(-1,-1){5.00}}
\put(35,45){\line(-1,1){5.00}}
\put(35,45){\line(1,0){5.00}}
\put(25,45){\circle{1}}
\put(30,40){\circle{1}}
\put(35,45){\circle{1}}
\put(30,50){\circle{1}}
\put(32,46){\circle{1}}
\put(28,46){\circle{1}}
\put(30,50){\line(-1,-2){2.00}}
\put(30,50){\line(1,-2){2.00}}
\put(28,46){\line(1,0){4.00}}
\put(40,45){\line(1,-1){5.00}}
\put(40,45){\line(1,1){5.00}}
\put(50,45){\line(-1,-1){5.00}}
\put(50,45){\line(-1,1){5.00}}
\put(40,45){\circle{1}}
\put(45,40){\circle{1}}
\put(50,45){\circle{1}}
\put(45,50){\circle{1}}
\put(47,46){\circle{1}}
\put(43,46){\circle{1}}
\put(45,50){\line(-1,-2){2.00}}
\put(45,50){\line(1,-2){2.00}}
\put(43,46){\line(1,0){4.00}}
\end{picture}\hspace{-15mm}
\unitlength 1.8mm
\begin{picture}(40,30)(0,0)
\put(7,10.0){\emph{b})}
\put(18,16.0){$\cG_*$}
\bezier{200}(15,15)(25,10)(35,15)
\bezier{300}(15,15)(-1,25)(15,34)
\bezier{300}(35,15)(51,25)(35,34)
\bezier{200}(15,34)(25,38.5)(35,34)

\bezier{300}(18,25)(25,20)(32,25)
\bezier{300}(20,24)(25,28)(30,24)

\put(25,14){\circle*{0.5}}
\put(25,21){\circle*{0.5}}
\put(28.5,17.5){\circle*{0.5}}
\put(21.5,17.5){\circle*{0.5}}
\put(25,21){\line(1,-1){3.5}}
\put(25,21){\line(-1,-1){3.5}}
\put(25,14){\line(1,1){3.5}}
\put(25,14){\line(-1,1){3.5}}

\put(25,21){\line(1,-2){1.5}}
\put(25,21){\line(-1,-2){1.5}}
\put(23.5,18){\circle*{0.5}}
\put(26.5,18){\circle*{0.5}}
\put(23.5,18){\line(1,0){3}}

\bezier{200}(16,20)(9,25)(16,29)
\bezier{200}(34,20)(41,25)(34,29)
\bezier{300}(16,29)(25,34)(34,29)
\bezier{150}(16,20)(19,18)(21.5,17.5)
\bezier{150}(28.5,17.5)(31,18)(34,20)

\put(20.0,18.3){\small$v_1$}
\put(27.8,18.3){\small$v_3$}
\put(22.3,13.8){\small$v_2$}
\put(22.5,21){\small$v_4$}
\put(22.5,16.5){\small$v_5$}
\put(25.5,16.5){\small$v_6$}

\bezier{40}(27,12.6)(33,18)(27,22.7)
\bezier{100}(25,21)(26,22.4)(27,22.7)
\bezier{100}(25,14)(26,12.7)(27,12.6)
\end{picture}
\vspace{-1cm} \caption{ \footnotesize  \emph{a}) A $\Z^2$-periodic graph $\cG\ss\R^2$; $\mathfrak{e}_1,\mathfrak{e}_2$ are the standard basis of $\R^2$; the unit cell $\Omega=[0,1)^2$ is shaded; \emph{b}) the quotient graph $\cG_*=\cG/\Z^2$ is a graph on the 2-dimensional torus $\R^2/\Z^2$; the torus is obtained by identification of opposite sides of $\Omega$.}
\lb{FBri}
\end{figure}

The following example illustrates Theorem \ref{MTh} and Remark \ref{reMT}.\emph{ii}.

\begin{example}\lb{E1DL}
We consider the Schr\"odinger operator $H_{\mu Q}$ on the $\Z^2$-periodic graph $\cG$ shown in  Fig.\ref{FBri}\emph{a}. Its quotient graph $\cG_*=(\cV_*,\cE_*)=\cG/\Z^2$ is a graph on the 2-dimensional torus $\R^2/\Z^2$, see Fig.\ref{FBri}\emph{b}, where the vertex set $\cV_*=\{v_1,\ldots,v_6\}$. The set $\cC^+$ of all homotopically nontrivial cycle in $\cG_*$ consists of the cycles
$$
(v_1,v_2,v_3,v_1),\qqq (v_1,v_4,v_3,v_1),\qqq (v_2,v_3,v_4,v_2),\qqq (v_2,v_1,v_4,v_2),
$$
of length 3 (given by the sequences of their vertices). The cycles
$(v_1,v_2,v_3,v_4,v_1)$ and $(v_4,v_5,v_6,v_4)$ are trivial, since they can be contracted to a point.

A $\Z^2$-periodic potential $Q$ is uniquely determined by the vector $(Q(v_j))_{j\in\N_6}$ of its values at the vertices of $\cG_*$. The large coupling asymptotics \er{Aswb} and \er{mare} for widths of the spectral bands and for the measure of the spectrum of $H_{\mu Q}$ for some potentials $Q$ are presented in Table \ref{TaMe}. For each $a\in Q(\cV_*)$, $d_j^a:=d(v_j,\cV_*^a)$ is the distance from the vertex $v_j$, $j\in\N_6$, to the set $\cV_*^a$ defined by \er{cVa}, $w(a)$ is given in \er{ooa}; $\s_{\mu a}$ is a spectral band "near" $\mu a$.

\begin{table}[h]
\begin{tabular}{|c|c|c|c|c|c|c|c|c|c|c|}
  \hline
  $Q=(Q(v_j))_{j=1}^6$ & $a$ & $d_1^a$ & $d_2^a$ & $d_3^a$ & $d_4^a$ & $d_5^a$ & $d_6^a$ & $\o(a)$ & $|\s_{\mu a}|$ & $|\s(H_{\mu Q})|$ \\
  \hline
  $(1,1,1,0,0,0)$ & $1$ & 0 & 0 & 0 & 1 & 2 & 2 & 0 & $O(1)$ & $O(1)$ \\
  \cline{2-10}
                  & $0$ & 1 & 1 & 1 & 0 & 0 & 0 & 2 & $O(\mu^{-2})$ &  \\\hline
  $(1,1,0,0,0,0)$ & $1$ & 0 & 0 & 1 & 1 & 2 & 2 & 1 & $O(\mu^{-1})$ & $O(\mu^{-1})$ \\
  \cline{2-10}
                  & $0$ & 1 & 1 & 0 & 0 & 0 & 0 & 1 & $O(\mu^{-1})$ &  \\\hline
  $(1,2,3,4,0,0)$ & $1$ & 0 & 1 & 1 & 1 & 2 & 2 & 2 & $O(\mu^{-2})$ & $O(\mu^{-2})$ \\
  \cline{2-10}
                  & $2$ & 1 & 0 & 1 & 1 & 2 & 2 & 2 & $O(\mu^{-2})$ &  \\
  \cline{2-10}
                  & $3$ & 1 & 1 & 0 & 1 & 2 & 2 & 2 & $O(\mu^{-2})$ &  \\
  \cline{2-10}
                  & $4$ & 1 & 1 & 1 & 0 & 1 & 1 & 2 & $O(\mu^{-2})$ &  \\
  \cline{2-10}
                  & $0$ & 2 & 2 & 2 & 1 & 0 & 0 & 4 & $O(\mu^{-4})$ &  \\\hline
\end{tabular}
\vspace{2mm}
\caption{\footnotesize  Measure of the spectrum of $H_{\mu Q}$ as $\mu\to\iy$; $d_j^a:=d(v_j,\cV_*^a)$, $a\in Q(\cV_*)$.}
\lb{TaMe}
\end{table}
\end{example}

\begin{remark}
For any $\Z^2$-periodic potential $Q$ such that $Q(v_5)=Q(v_6)=0$ the Schr\"odinger operator $H_{\mu Q}$ on the periodic graph $\cG$ shown in  Fig.\ref{FBri}\emph{a} has the degenerate band $\{-1\}$ (the eigenvalue of infinite multiplicity) with an eigenfunction $\vp$ supported on $\{v_5,v_6\}$, $\vp(v_5)=-\vp(v_6)=1$.
\end{remark}

\medskip

The remainder of the paper is organized as follows. Section \ref{Sec2} contains definitions of the discrete Schr\"odinger operators on periodic graphs, quotient graphs and their homotopically nontrivial cycles, and a brief description of the Floquet reduction and spectra of the Schr\"odinger operators. Section 3 is devoted to the proof of Theorem \ref{MTh}.

\section{Preliminaries}\lb{Sec2}
\subsection{Periodic graphs and edge indices.}
Let $\cG=(\cV,\cE)$ be a simple (i.e., without loops and multiple edges) connected infinite graph, where $\cV$ is the set of its vertices and $\cE$ is the set of its unoriented edges. Considering each edge in $\cE$ to have two orientations, we introduce the set $\cA$ of all oriented edges. An edge starting at a vertex $u$ and ending at a vertex $v$ from $\cV$ will be denoted as the ordered pair $(u,v)\in\cA$. Vertices $u,v\in\cV$ will be called \emph{adjacent} and denoted by $u\sim v$, if $(u,v)\in\cA$.

We consider \emph{locally finite $\Z^d$-periodic graphs} $\cG$, i.e., graphs satisfying the following conditions:
\begin{itemize}
  \item[1)] $\cG$ is equipped with a free action of the group $\Z^d$;
  \item[2)] the quotient graph $\cG_*=\cG/\Z^d$ is finite.
\end{itemize}
The graph $\cG_*=(\cV_*,\cE_*)$ has the vertex set $\cV_*=\cV/\Z^d$, the set $\cE_*=\cE/\Z^d$ of unoriented edges and the set $\cA_*=\cA/\Z^d$ of doubled oriented edges.

Following \cite{F26}, we assume that the quotient graph $\cG_*$ has no loops and multiple edges. This assumption is not essential (for additional explanations see \cite[Section 2.1]{F26}), but simplifies the analysis of periodic graph operators. It is also convenient to embed a $\Z^d$-periodic graph $\cG$ into $\R^d$ in such a way that it is invariant with respect to shifts along any $\mn\in\Z^d$. In this case the  quotient graph $\cG_*=\cG/\Z^d$ is a graph on the $d$-dimensional torus $\R^d/\Z^d$, see Fig.\ref{FBri}.

We define the important notion of an {\it edge index} which was introduced in \cite{KS14}. This notion allows one to consider, instead of a periodic graph, the finite quotient graph with edges labeled by some integer vectors called indices.

For any vertex $v\in\cV$ of a $\Z^d$-periodic graph $\cG$, the following  unique representation holds true:
\[\lb{Dv}
v=v_0+[v], \qqq \textrm{where}\qqq v_0\in\cV\cap[0,1)^d,\qqq [v]\in\Z^d.
\]
For any oriented edge $\be$ of the periodic graph $\cG$, we define the \emph{edge index}
\[\lb{in}
\t(\be)=[v]-[u]\in\Z^d,\qqq \be=(u,v)\in\cA,
\]
where $[v]\in\Z^d$ is given by \er{Dv}. Due to the $\Z^d$-periodicity of the graph $\cG$, the edge indices $\t(\be)$ satisfy
\[\lb{Gpe}
\t(\be+\mn)=\t(\be),\qqq \forall\, (\be,\mn)\in\cA \ts\Z^d.
\]
This periodicity of the indices allows us to assign the \emph{index} $\t(\be_*)\in\Z^d$ to each oriented edge $\be_*\in\cA_*$ of the quotient graph $\cG_*$ by setting
\[\lb{dco}
\t(\be_*)=\t(\be),
\]
where $\be\in\cA$ is an oriented edge in the periodic graph $\cG$ from the equivalence class $\be_*\in\cA_*=\cA/\Z^d$. In other words, edge indices of the quotient graph $\cG_*$ are induced by edge indices of the periodic graph~$\cG$. Due to (\ref{Gpe}), the edge index $\t(\be_*)$ is uniquely determined by \er{dco} and does not depend on the choice of $\be\in\cA$. From the definition of the edge indices it follows that
\[\lb{inin}
\t(\ol\be\,)=-\t(\be), \qqq \forall\,\be\in\cA_*,
\]
where $\ol\be=(v,u)$ is the inverse edge of $\be=(u,v)\in\cA_*$.

\subsection{Cycles and their indices} \lb{ssci}
A \emph{path} $\gp$ in a graph $\cG=(\cV,\cA)$ is a sequence of oriented edges
$$
\gp=(\be_1,\be_2,\ldots,\be_l), \qqq \textrm{where} \qqq
\be_j=(v_{j-1},v_{j})\in\cA,\qq j=1,\ldots,l,
$$
for some vertices $v_0,v_1,\ldots,v_l\in\cV$. The number $l$ of edges in a path $\gp$ is called the \emph{length} of $\gp$ and is denoted by $|\gp|$, i.e., $|\gp|=l$. A path $\gp$ is called \emph{closed}, if $v_0=v_l$. A \emph{cycle} $\bc$ is a closed path of length $l\geq 3$ with distinct intermediate vertices.

\begin{remark}\lb{Rcdv} If a graph $\cG$ has no multiple edges, then a path $\gp$ can equivalently be described by the sequence $\gp=(v_0,v_1,\ldots,v_l)$ of vertices it passes by.
\end{remark}

For any path $\gp$ in the quotient graph $\cG_*=(\cV_*,\cA_*)$, we define the \emph{index} $\t(\gp)\in\Z^d$ as
\[\lb{cyin}
\t(\gp)=\t(\be_1)+\ldots+\t(\be_l),  \qqq \gp=(\be_1,\ldots,\be_l),
\]
where $\t(\be)\in\Z^d$ is the index of the edge $\be\in\cA_*$.

\begin{remark}\lb{Rein}
Edge indices depend on the choice of the embedding of the periodic graph $\cG$ into $\R^d$ (see Fig.1 in \cite{KS22}). Cycle indices \emph{do not} depend on this choice. Indeed, any cycle $\bc$ in the quotient graph $\cG_*$ is obtained by factorization of a path in the periodic graph $\cG$ connecting some $\Z^d$-equivalent vertices $v\in\cV$ and $v+\mn\in\cV$, $\mn\in\Z^d$, and the cycle index $\t(\bc)=\mn$, i.e., it does not depend on the choice of the embedding. In particular, $\t(\bc)=0$ if and only if the cycle $\bc$ in $\cG_*$ corresponds to a cycle in~$\cG$.
\end{remark}

Cycles in the quotient graph $\cG_*$ with non-zero indices are \emph{homotopically nontrivial} cycles on the $d$-dimensional torus $\R^d/\Z^d$, since they cannot be contracted to points, see Fig.\ref{FBri}.

\subsection{Spectra of the Schr\"odinger operators on periodic graphs.}
Let $\ell^2(\cV)$ be the Hilbert space of all square summable
functions  $f:\cV\to \C$ equipped with the norm
$$
\|f\|^2_{\ell^2(\cV)}=\sum_{v\in\cV}|f(v)|^2<\infty.
$$

We consider the discrete Schr\"odinger operator $H_Q=\D+Q$ acting on $\ell^2(\cV)$, where $\D$ is the \emph{adjacency} operator having the form
\[\lb{ALO}
(\D f)(v)=\sum_{u\sim v}f(u), \qqq f\in\ell^2(\cV), \qqq v\in\cV,
\]
and $Q$ is a real $\Z^d$-periodic potential, i.e., it satisfies
$$
(Qf)(v)=Q(v)f(v), \qqq Q(v+\mn)=Q(v), \qqq \forall\,(v,\mn)\in\cV\ts\Z^d.
$$
The sum in \er{ALO} is taken over all vertices $u\in\cV$ adjacent to the
vertex $v$.

The spectral analysis of periodic operators is based on the Floquet-Bloch theory (see, e.g., \cite{RS78} or, for the graph case, \cite{BK13}, Chapter 4). According to this theory, the spectrum of the Schr\"odinger operator $H_Q$ on a periodic graph $\cG$ is given by
$$
\s(H_Q)=\bigcup_{k\in\T^d}\s\big(H_Q(k)\big),\qqq \T^d=\R^d/(2\pi\Z^d).
$$
Here the \emph{Floquet $\n\ts\n$ matrix} $H_Q(k)$, $\n=\#\cV_*$, has the form
\[\lb{Flop}
H_Q(k)=\D(k)+Q,\qqq Q=\diag(Q(v))_{v\in\cV_*},
\]
where
\[\label{Hvt'}
\D(k)=(\D_{uv}(k))_{u,v\in\cV_*},\qqq \D_{uv}(k)=\left\{
\begin{array}{cl}
e^{i\lan\t(u,v),k\ran}, \;& \textrm{if}\;(u,v)\in\cA_*, \\[4pt]
0, \; & \textrm{otherwise},
\end{array}\right.
\]
$\t(u,v)\in\Z^d$ is the index of the edge $(u,v)\in\cA_*$, and $\lan\cdot,\cdot\ran$ denotes the standard inner product in $\R^d$. For more details and proof see \cite{KS14}.

Each Floquet matrix $H_Q(k)$, $k\in\T^{d}$, is self-adjoint and has $\n$ real eigenvalues
$$
\l_{1}(k)\leq\l_{2}(k)\leq\ldots\leq\l_{\nu}(k),
$$
labeled in non-decreasing order counting multiplicities.
Each \emph{band function} $\l_j(\cdot)$ is continuous and piecewise real analytic on the torus $\T^d$ and creates the \emph{spectral band} $\s_j(H_Q)$ given by
\[\lb{ban.1H}
\begin{array}{l}
\s_j(H_Q)=[\l_j^-,\l_j^+]=\l_j(\T^{d}),\qqq j\in\N_\n=\{1,\ldots,\n\},\\[8pt]
\displaystyle\textrm{where}\qqq \l_j^-=\min_{k\in\T^d}\l_j(k),\qqq \l_j^+=\max_{k\in\T^d}\l_j(k).
\end{array}
\]
Thus, the spectrum of the Schr\"odinger operator $H_Q$ on a periodic graph $\cG$ has the form
$$
\s(H_Q)=\bigcup_{k\in\T^d}\s\big(H_Q(k)\big)=
\bigcup_{j=1}^{\nu}\s_j(H_Q),
$$
i.e., it consists of $\n$ bands $\s_j(H_Q)$ defined by \er{ban.1H}.

\begin{remark}
Some of spectral bands $\s_j(H_Q)=[\l_j^-,\l_j^+]$ may be degenerate, i.e., $\l_j^-=\l_j^+$.
\end{remark}

\section{Proof of the main result}
\setcounter{equation}{0}
\subsection{Minimum spanning trees}\lb{sec3.1}
We recall the definition of spanning trees and their properties which will be used in the proof of our results. Let $G=(V,E)$ be a finite simple connected graph.

A \emph{spanning tree} $T=(V,E_{T})$ of the graph $G$ is a connected subgraph of $G$ which has no cycles and contains all vertices of $G$. We introduce the set $\cS=E\sm E_T$ of all edges from $E$ that do not belong to the spanning tree $T$. It is known (see, e.g., \cite[Lemma 5.1, Theorem~5.2]{B74}) that

$\bu$ the set $\cS$ consists of exactly $\beta$ edges, where $\beta=\#E-\#V+1$ is the Betti number of the graph $G$;

$\bu$ for each $\be\in\cS$ there exists a unique cycle $\bc_\be$ in $G$ whose edges are all in $T$ except $\be$;

$\bu$ the set of all such cycles $\bc_\be$, $\be\in\cS$, forms a basis of the cycle space of the graph~$G$.

\medskip

With edges of $G=(V,E)$ we assign a weight $\o:E\to\R$ and consider the weighted graph $G=(V,E,\o)$.

A \emph{minimum} spanning tree (MST) of the weighted graph $G=(V,E,\o)$ is a spanning tree of $G$ with the minimum possible total edge weight. The graph $G$ may have several MST of the same weight.

\emph{Cycle property of a MST}: for any cycle $\bc$ in the graph $G=(V,E,\o)$, the edge of $\bc$ with the maximum weight (or one of them if there are ties) is not included in the MST.

\subsection{A choice of the periodic graph embedding} In this section, for each $a\in Q(\cV_*)$, we define an appropriate embedding of the periodic graph $\cG$ into $\R^d$ which allows us to estimate the width of the spectral bands near $\mu a$.

\begin{proposition}\lb{Ppge}
Let $a\in Q(\cV_*)$, and let the weight $\o_a(\be)$ of each edge $\be\in\cE_*$ of the quotient graph $\cG_*=(\cV_*,\cE_*)$ be defined by \er{wae}. Then there is a 1-form $\t_a:\cA_*\to\Z^d$ such that

$\bu$ for each $k\in\T^d$, the Floquet operators $H_Q(k)$ given by \er{Flop}, \er{Hvt'}, with the index 1-form $\t:\cA_*\to\Z^d$, defined by \er{in}, \er{dco}, and the 1-form $\t_a$ are unitarily equivalent, and

$\bu$ for all $\be\in\cA_*$ with $\t_a(\be)\neq0$ we have
\[\lb{oaeg}
\o_a(\be)\geq \o(a),\qqq \o(a):=\min\limits_{\bc\in\cC^+}\max\limits_{\be\in\bc}\o_a(\be).
\]
Here $\cC^+$ is the set of all cycles with non-zero indices in $\cG_*$, and the maximum is taken over all edges $\be$ of the cycle $\bc$.
\end{proposition}

\no \textbf{Proof.} Let $a\in Q(\cV_*)$, and let $T_a=(\cV_*,\cE_a)$ be a minimum spanning tree of the weighted quotient graph $\cG_*=(\cV_*,\cE_*,\o_a)$. Recall that for each $\be\in\cS_a:=\cE_*\sm\cE_a$ there exists a unique cycle $\bc_\be$ in $\cG_*$ whose edges are all in $T_a$ except $\be$.

We equip each edge of $\cS_a$ with some orientation and associate with $T_a$ a 1-form $\t_a:\cA_*\to\Z^d$ equal to zero on edges of $T_a$ and coinciding with the index of the cycle $\bc_\be$ on each edge $\be\in\cS_a$: 
\[\lb{cat}
\t_a(\be)=\left\{
\begin{array}{cl}
\t(\bc_\be), &  \textrm{ if } \, \be\in\cS_a\\[2pt]
-\t(\bc_{\ol\be}), &  \textrm{ if } \, \ol\be\in\cS_a\\[2pt]
  0, \qq & \textrm{ otherwise} \\
\end{array}\right..
\]
Due to the definition, this 1-form $\t_a:\cA_*\to\Z^d$ preserves indices of all basis cycles $\bc_\be$, $\be\in\cS_a$, of $\cG_*$. Consequently, it preserves indices of all cycles in $\cG_*$. Then \cite[Theorem 4.4.\emph{ii}]{KS20} for each $k\in\T^d$ the Floquet operators $H_Q(k)$ with the 1-forms $\t$ and $\t_a$ are unitarily equivalent.

Let $\be_o\in\cE_*$ and $\t_a(\be_o)\neq0$. We will show that $\o_a(\be_o)\geq\o(a)$, where $\o(a)$ is defined in \er{oaeg}. Since $\t_a(\be)=0$ for each edge $\be\in\cE_a$, $\be_o\notin\cE_a$. Then there is a (unique) cycle $\bc_{\be_o}$ consisting of the edge $\be_o$ and edges of $T_a$. Due to the \emph{cycle property of a MST} (see section \ref{sec3.1}), we have $\o_a(\be)\leq\o_a(\be_o)$ for all $\be\in\bc_{\be_o}$. Then, using that $\bc_{\be_o}\in\cC^+$, we obtain
$$
\o(a)=\min\limits_{\bc\in\cC^+}\max\limits_{\be\in\bc}\o_a(\be)\leq
\max\limits_{\be\in\bc_{\be_o}}\o_a(\be)=\o_a(\be_o).
$$
Thus, the 1-form $\t_a$ given by \er{cat} satisfies the condition \er{oaeg}. \qq $\Box$

\begin{remark}\lb{Remb}
For each $a\in Q(\cV_*)$, the 1-form $\t_a:\cA_*\to\Z^d$ given by \er{cat} determines the embedding of the periodic graph $\cG$ into $\R^d$, such that the index form $\t$, defined by \er{in}, \er{dco}, coincides with $\t_a$, see \cite[Corollary 2.3]{KS20}. The spectrum of the Schr\"odinger operator $H_Q$ on $\cG$ does not depend on the choice of the embedding.
\end{remark}

\subsection{Upper bounds on the spectral bands widths}

We obtain an upper bound on the width of the spectral bands of the Schr\"odinger operator $H_Q$ on a periodic graph $\cG$ with the quotient graph $\cG_*=(\cV_*,\cA_*)$.

\begin{proposition}\lb{Pubb}
For each $n\in\N_\n$, $\n=\#\cV_*$, there exists $k_o\in\T^d$ (depending on $n$) such that the width of the n-th spectral band $\s_n(H_Q)$ of the Schr\"odinger operator $H_Q$ is bounded by
\[\lb{upbu}
\big|\s_n(H_Q)\big|\leq2\pi d\sum_{\substack{\be=(u,v)\in\cA_*\\\t(\be)\neq0}}
\big\|\t(\be)\big\|\cdot\big|\vp_n(k_o,u)\big|\cdot\big|\vp_n(k_o,v)\big|,
\]
where $\t(\be)\in\Z^d$ is the index of the edge $\be\in\cA_*$, $\|\cdot\|$ is the standard norm in $\R^d$, and $\vp_n(k_o)=\vp_n(k_o,\cdot)\in\ell^2(\cV_*)$ is a normalized solution of $H_Q(k_o)\vp_n(k_o)=\l_n(k_o)\vp_n(k_o)$.
\end{proposition}

\no \textbf{Proof.} Let $\s_n(H_Q)=[\l_n^-,\l_n^+]$, where $\l_n^\pm=\l_n(k^\pm)$ for some $k^\pm\in\T^d$.
We define the family of operators
\[\lb{A(t)}
A(t)=H_Q(k(t)), \qqq k(t):=k^-+t(k^+-k^-),\qqq t\in[0,1],
\]
which is analytic. The eigenvalues of this one-parameter family can
be represented as a set of analytic functions $E_l(t)$, $l\in\N_\n$ (that can intersect). We have
\[\lb{deEt}
E_l'(t)=\lan\psi_l(t),A'(t)\psi_l(t)\ran,
\]
where $\psi_l(t)\in\ell^2(\cV_*)$ is a normalized solution of $A(t)\psi_l(t)=E_l(t)\psi_l(t)$. Since
$$
\l_n\big(k(t)\big)=E_l(t)\qqq \textrm{for some } l=l(t), \qqq t\in[0,1],
$$
then $\dfrac{d\l_n(k(t))}{dt}=E'_l(t)$ except at finitely many points.
Using the mean value theorem and the identities
$$
\l_n(k(0))=\l_n(k^-)=\l_n^-,\qqq \l_n(k(1))=\l_n(k^+)=\l_n^+,
$$
and \er{deEt}, we obtain that for some $t_o\in(0,1)$, 
\[\lb{leba}
\l_n^+-\l_n^-\leq\big|E'_{l_o}(t_o)\big|=
\big|\lan\psi_{l_o}(t_o),A'(t_o)\psi_{l_o}(t_o)\ran\big|,\qqq l_o:=l(t_o).
\]
Due to \er{A(t)}, we have
$$
A'(t)=\sum_{j=1}^d(k_j^+-k_j^-)\,\frac{\pa H_Q(k)}{\pa k_j}\Big|_{k=k(t)},\qqq
k=(k_1,\ldots,k_d).
$$
Then, using \er{Flop}, \er{Hvt'}, the edge index property \er{inin}, 
and the shorthand notation $\psi:=\psi_l(t)$, we obtain
\begin{multline*}
\lan\psi,A'(t)\psi\ran=\sum_{j=1}^d(k_j^+-k_j^-)\big\lan\psi,
\frac{\pa H_Q(k)}{\pa k_j}\Big|_{k=k(t)}\psi\big\ran\\=
i\sum_{j=1}^d(k_j^+-k_j^-)\sum_{\substack{\be=(u,v)\in\cA_*\\\t_j(\be)\neq0}}\t_j(\be) e^{-i\lan\t(\be),k(t)\ran}\psi(u)\ol{\psi(v)}\\
=\frac i2\sum_{j=1}^d(k_j^+-k_j^-)
\sum_{\substack{\be=(u,v)\in\cA_*\\\t_j(\be)\neq0}}\t_j(\be)\Big(e^{- i\lan\t(\be),k(t)\ran}\psi(u)\ol{\psi(v)}-
e^{i\lan\t(\be),k(t)\ran}\ol{\psi(u)}\psi(v)\Big)\\=
\sum_{j=1}^d(k_j^+-k_j^-)\sum_{\substack{\be=(u,v)\in\cA_*\\\t_j(\be)\neq0}}\t_j(\be)
\Im\big(e^{i\lan\t(\be),k(t)\ran}\ol{\psi(u)}\psi(v)\big),
\end{multline*}
where $\t(\be)=\big(\t_1(\be),\ldots,\t_d(\be)\big)\in\Z^d$.
Since $k^\pm\in\T^d$, we have $|k_j^+-k_j^-|\leq2\pi$ for each $j\in\N_d$. Then
$$
\big|\lan\psi,A'(t)\psi\ran\big|\leq
2\pi\sum_{j=1}^d\sum_{\substack{\be=(u,v)\in\cA_*\\\t_j(\be)\neq0}}
\hspace{-2mm}\big|\t_j(\be)\big|\cdot\big|\psi(u)\big|\cdot\big|\psi(v)\big|\leq
2\pi d\hspace{-3mm}\sum_{\substack{\be=(u,v)\in\cA_*\\\t(\be)\neq0}}
\hspace{-2mm}
\big\|\t(\be)\big\|\cdot\big|\psi(u)\big|\cdot\big|\psi(v)\big|.
$$
Substituting this into \er{leba}, we obtain
\[\lb{leb}
\l_n^+-\l_n^-\leq2\pi d\sum_{\substack{\be=(u,v)\in\cA_*\\\t(\be)\neq0}}
\big\|\t(\be)\big\|\cdot\big|\psi_{l_o}(t_o,u)\big|\cdot
\big|\psi_{l_o}(t_o,v)\big|.
\]
Since $\psi_{l_o}(t_o)$ is a normalized eigenvector of $H_Q(k_o)$ corresponding to $\l_n\big(k_o\big)$, where $k_o=k(t_o)\in\T^d$ for $k(\cdot)$ defined by \er{A(t)}, the inequality \er{upbu} follows from \er{leb}. \qq $\Box$

\medskip

\no \textbf{Proof of Theorem \ref{MTh}.} Let $\cV_*=\{1,\ldots,\n\}$ be the vertex set of the quotient graph $\cG_*=(\cV_*,\cA_*)$. Let $a\in Q(\cV_*)$ and $m_a=\#\cV_*^a$, where $\cV_*^a$ is defined by \er{cVa}. Then there are exactly $m_a$ bands
$$
\s_s(H_{\mu Q}), \qqq j_o\leq s\leq j_o+m_a-1 \qq \textrm{(for some $j_o=j_o(a)\in\N_\n$)},
$$
of $H_{\mu Q}$ that arise from $\mu a$. We will estimate the width of each of these bands.

Let $\ve=\frac1\mu$\,. For each fixed $k\in\T^d$, the matrix $A(\ve,k):=Q+\ve\D(k)$ has the same eigenvectors as the Floquet matrix $H_{\mu Q}(k)=\D(k)+\mu Q$ defined by \er{Flop}, \er{Hvt'}. In the limit $\ve\to0$ ($\mu\to\iy$) these eigenvectors can be obtained from perturbation theory considering $\ve\D(k)$ as a perturbation of $Q$. Let $\{\mathfrak{e}_s : s\in\N_\n\}$ be the standard basis of $\C^\n$.

First, we consider the case when $m_a=1$, i.e., $a=q_n:=Q(n)$ for a unique vertex $n\in\cV_*$. In this case, the normalized eigenvector of $Q$ corresponding to $a$ is $\mathfrak{e}_n$. Then the eigenvector of $A(\ve,k)$ corresponding to $\mathfrak{e}_n$, and to the eigenvalue $a$ of $Q$, is given by
\[\lb{vpk}
\vp(k)\equiv\big(\vp(k,1),\ldots,\vp(k,\n)\big)^\top=
P(\ve,k)\mathfrak{e}_n, \qqq \|\vp(k)\|=1+O(\ve),
\]
where $P(\ve,k)=\big(P_{ij}(\ve,k)\big)_{i,j=1}^\nu$ is the spectral projection given by
\[\lb{Pnth}
P(\ve,k)=\frac1{2\pi i}\oint\limits_{|\l-a|=\epsilon}\big(\l-A(\ve,k)\big)^{-1}d\l
\]
with a small $\epsilon>0$. We have
\[\lb{res}
\big(\l-A(\ve,k)\big)^{-1}=\sum_{r=0}^\iy\ve^r R^{(r)}(\l,k),
\]
where
$$
R^{(r)}(\l,k)=(\l-Q)^{-1}\big(\D(k)(\l-Q)^{-1}\big)^r=
\big(R^{(r)}_{ij}(\l,k)\big)_{i,j=1}^\nu
$$
is the $\n\ts\n$ matrix with entries
\[\lb{Rrij}
R^{(r)}_{ij}(\l,k)=\sum_{i_1,\ldots,i_{r-1}=1}^\n
\frac{\D_{ii_1}(k)\D_{i_1i_2}(k)\ldots\D_{i_{r-1}j}(k)}
{(\l-q_i)(\l-q_{i_1})\ldots(\l-q_{i_{r-1}})(\l-q_j)}\,.
\]
Substituting the second identity from \er{Hvt'} into \er{Rrij} and using the definition \er{cyin} of the path index $\t(\gp)$, we obtain
\[\lb{Rrij1}
R^{(r)}_{ij}(\l,k)=\sum_{\substack{\gp:i\to j\\|\gp|=r}}
\frac{e^{i\lan\t(\gp),k\ran}}
{\prod\limits_{l\in\gp}(\l-q_l)}\,,
\]
where $\gp:i\to j$ is a path of length $r$ from $i$ to $j$, and the product is taken over all vertices $l$ of the path $\gp$. Using \er{vpk} -- \er{res} and \er{Rrij1}, we have
\[\lb{vpnd}
\begin{aligned}
& \vp(k,n)=P_{nn}(\ve,k)=1+O(\ve),\\
& \vp(k,i)=P_{in}(\ve,k)=\ve^{d(i,n)}\sum_{\substack{\gp:i\to n\\|\gp|=d(i,n)}}\frac{e^{\lan\t(\gp),k\ran}}
{\prod\limits_{l\in\gp\sm\{n\}}(q_n-q_l)}+O\big(\ve^{d(i,n)+1}\big),\qqq i\neq n,
\end{aligned}
\]
where $d(i,n)$ is the distance between the vertices $i$ and $n$.

Next, we consider the case when $m_a=2$ (the case of arbitrary $m_a\geq3$ works out exactly the same way). Let $\cV_*^a=\{n,n'\}$. Normalized eigenvectors of $A(\ve,k)$ corresponding to the eigenvalue $a$ of $Q$, are given by
\[\lb{vpjk}
\vp^{(j)}(k)=P(\ve,k)(b_j\mathfrak{e}_n+b'_j\mathfrak{e}_{n'}), \qqq j=1,2,
\]
where $b_j$ and $b'_j$ are the solutions of the system
$$
\begin{array}{l}
b_j\lan\vp^{(j)}(k),\mathfrak{e}_n\ran+b'_j\lan\vp^{(j)}(k),\mathfrak{e}_{n'}\ran=1,
\\[3pt]
b_j\lan\vp^{(t)}(k),\mathfrak{e}_n\ran+b'_j\lan\vp^{(t)}(k),\mathfrak{e}_{n'}\ran=0,
\end{array}\qqq
t=\left\{\begin{array}{cc}
           1, & j=2 \\[2pt]
           2, & j=1
         \end{array}\right..
$$
Since $b_j$ and $b'_j$ are both (in general) $O(1)$, using \er{vpjk}, \er{Pnth}, \er{res} and \er{Rrij1}, we obtain
\[\lb{vpdc}
\vp^{(j)}(k,i)=b_jP_{in}(\ve,k)+b'_jP_{in'}(\ve,k)=
\left\{\begin{array}{ll}
O(1), & i=n,n',\\[4pt]
O\big(\ve^{d(i,\cV_*^a)}\big), \qq & i\neq n,n',
\end{array}\right.
\]
where $d(i,\cV_*^a)$ is the distance from the vertex $i$ to the vertex set $\cV_*^a=\{n,n'\}$ defined by \er{dvse}.

Due to \er{vpnd} and \er{vpdc}, we obtain that for each $k\in\T^d$ the eigenvectors $\vp_s(k)$ corresponding to the eigenvalues $\l_s(k)$ of the Floquet matrix $H_{\mu Q}(k)$ satisfy
\[\lb{evas}
\vp_s(k,v)=O\big(\ve^{d(v,\cV_*^a)}\big), \qqq v\in\cV_*,\qqq j_o\leq s\leq j_o+m_a-1.
\]

By Proposition \ref{Pubb}, for each $s\in\N_\n$, there exists $k_o\in\T^d$ (depending on $s$) such that
\[\lb{snal}
\big|\s_s(H_{\mu Q})\big|\leq2\pi d\sum_{\substack{\be=(u,v)\in\cA_*\\\t(\be)\neq0}}
\big\|\t(\be)\big\|\cdot
\big|\vp_s(k_o,u)\big|\cdot\big|\vp_s(k_o,v)\big|.
\]
Using \er{evas}, we have that for each $s=j_o,\ldots,j_o+m_a-1$ and for each $\be=(u,v)\in\cA_*$,
$$
\big|\vp_s(k_o,u)\big|\cdot\big|\vp_s(k_o,v)\big|=
O\big(\ve^{d(u,\cV_*^a)+d(v,\cV_*^a)}\big)=O\big(\ve^{\o_a(\be)}\big),
$$
where $\o_a(\be)$ is given by \er{wae}. We choose the embedding of $\cG$ into $\R^d$ such that the index form $\t=\t_a$, where $\t_a:\cA_*\to\Z^d$ is defined by Proposition \ref{Ppge} (see also Remark \ref{Remb}). Using that for each edge $\be\in\cA_*$ with $\t_a(\be)\neq0$ the condition \er{oaeg} holds, from \er{snal} we conclude that the width of each band $\s_s(H_{\mu Q})$ corresponding to $\mu a$ satisfies
$$
\big|\s_s(H_{\mu Q})\big|=O(\ve^{\o(a)}),
$$
where $\o(a)$ is defined in \er{oaeg}. The total measure of $\s(H_{\mu Q})$ has the order of magnitude of the widest bands. Therefore, we obtain
$$
|\s(H_{\mu Q})|=O(\ve^\o), \qqq \o=\min\limits_{a\in Q(\cV_*)}\o(a). 
$$
Since $\ve=\frac1\mu$\,, we get the asymptotics \er{mare} \qq $\Box$

\medskip

\textbf{Acknowledgments.}  This work is supported by the Russian Science Foundation (project No. 25-21-00157).

\medskip

\textbf{Data availability}  No datasets were generated or analysed during the current study.

\medskip

\textbf{Conflict of interest} None declared.


\end{document}